\documentclass{amsart}
\usepackage{amssymb}
\usepackage{amsthm}
\usepackage{verbatim}
\usepackage{hyperref}
\usepackage{color}
\newtheorem{thm}{Theorem}[section]
\newtheorem{cor}[thm]{Corollary}

\newtheorem{prop}[thm]{Proposition}
\newtheorem{exam}[thm]{Example}
\theoremstyle{definition}
\theoremstyle{remark}
\newtheorem{rem}[thm]{Remark}
\numberwithin{equation}{section}

\newcommand{\va}{\varphi}

\begin{document}

\title[subspace-hypercyclic conditional type operators  ]
{subspace-hypercyclic conditional type operators on $L^p$-spaces }

\author{\bf M. R. Azimi$^{*^1}$ and Z. Naghdi$^2$}
\address {M. R. Azimi}
\email{mhr.azimi@maragheh.ac.ir(for corresponding)}
\address{Z. Naqdi}
\email{z.nagdi1396@gmail.com}
\address{$1, 2$: department of mathematics, faculty of sciences,
university of maragheh, p. o. box: 55181-83111, Golshahr, maragheh,
iran}




\date{}

\dedicatory{}

\commby{}


\begin{abstract}
 A conditional weighted composition
 operator   $T_u: L^p(\Sigma)\rightarrow L^p(\mathcal{A})$ ($1\leq
p<\infty$),  is defined by $T_u(f):= E^{\mathcal{A}}(u f\circ
\varphi)$, where $\varphi: X\rightarrow X$ is a measurable
transformation, $u$ is a weight function on $X$ and
$E^{\mathcal{A}}$ is the conditional expectation operator with
respect to $\mathcal{A}$. In this paper, we study the
subspace-hypercyclicity of $T_u$ with respect to $L^p(\mathcal{A})$.
First, we show that if $\varphi$ is a periodic nonsingular
transformation, then $T_u$ is not $L^p(\mathcal{A})$-hypercyclic.
The necessary conditions for the subspace-hypercyclicity of $T_u$
are obtained when $\varphi$ is non-singular and finitely non-mixing.
For the sufficient conditions, the normality of  $\varphi$ is
required.
 The subspace-weakly mixing
and subspace-topologically mixing concepts are also studied for
$T_u$. Finally, we give an example which is subspace-hypercyclic
while is not hypercyclic.
\end{abstract}

\maketitle
\section{\textbf{Introduction and Preliminaries}}

 Suppose that $T$ is a bounded linear operator on a topological vector space
$X$. If there is a vector $x \in X$ such that the orbit $orb(T,
x):=\{T^n x : n=0, 1, 2,... \}$ is dense in $X$, then $T$ will be
hypercyclic and $x$ is called a hypercyclic vector. Here, $T^n$
stands for the $n$-th iterate of $T$ and $T^0$ is the identity map
$I$. Let $M$ be a closed and non-trivial subspace of $X$. An
operator $T$ is {\em subspace-hypercyclic} with respect to $M$
($M$-\emph{hypercyclic}), if there is a
   a vector $x\in X$   such that $orb(T, x)\cap M$
is dense in $M$. Also an operator $T$ is {\em subspace-transitive}
with respect to $M$, if for any non-empty open set $U, V\subseteq
M$, there exists an $n\in\mathbb{N}$ such that $T^{-n}(U)\cap V$
contains an open non-empty subset of $M$. An operator $T$ is {\em
subspace-topologically mixing} with respect to $M$, if for any
non-empty open set $U, V\subseteq M$, there exists an
$N\in\mathbb{N}$ such that $T^{-n}(U)\cap V$ contains an open
non-empty subset of $M$ for each $n \geq N$. It is called {\em
subspace-weakly mixing} if $T \oplus T$ is subspace-hypercyclic with
respect to  $M\oplus M$.

The study of subspace-hypercyclic linear operators  was initiated by
B. F. Madore and R. A. Mart\'{i}nez-Avenda\~{n}o \cite{mad}.
  They found out that
  subspace-hypercyclicity like as hypercyclicity,  can occur only on
  infinite-dimensional spaces and even subspaces. Also,  they
 proved an interesting Kitai's type  \emph{subspace-hypercyclicity
criterion} on a topological vector space as follows.\\
Assume that there exist $D_1$ and $D_2$, dense subsets of $M$, and
an increasing sequence of positive integers $(n_k)$ such that
\begin{itemize}
  \item $T^{n_k}x\to 0$ for all $x\in D_1$;
  \item for each $y\in D_2$, there exists a sequence $\{x_{k}\}$
 in $M$ such that $x_{k}\to 0$ and $T^{n_k}x_{k}\to y$;
  \item $M$ is an invariant subspace for $T^{n_k}$
  for all $k\in\mathbb{N}$.
\end{itemize}
Then $T$ is subspace-transitive and hence  is subspace-hypercyclic
\cite[Theorem 3.6]{mad}. But the converse is not true, see \cite{21,
rez2} for more details. Further, it is showed that the compact or
hyponormal operators are not subspace-hypercyclic.

For the dynamics of linear operators the  survey articles
\cite{sal}, \cite{be},  \cite{rez}, \cite{yus}, \cite{aba},
\cite{mad} and the books \cite{bay}, \cite{mang} are useful.

Let $(X,\Sigma,\mu)$ be a complete $\sigma$-finite measure space and
$\mathcal{A}$ is a $\sigma$-finite subalgebra of $\Sigma$. For each
$1\leq p < \infty$,  the Banach space
$L^p(X,\mathcal{A},\mu_{|_{\mathcal{A}}})$ is denoted by
$L^p(\mathcal{A})$ simply.  All comparisons between two functions or
two sets are to be interpreted as holding up to a $\mu$-null set.
The \emph{support} of any $\Sigma$-measurable function $f$ is
defined by $ \sigma(f)=\lbrace x\in X : f(x)\neq0 \rbrace$. The
\emph{characteristic function} of any set $A$ and  the class of all
$\mathcal{A}$-measurable and simple functions on $X$ with finite
supports will be denoted by $\chi_A$ and $S^{\mathcal{A}}(X)$,
respectively.

A $\Sigma$-measurable transformation $\va :X \rightarrow X$ is
called \emph{non-singular} whenever  $\mu \circ \va^{-1}$ is
absolutely continuous with respect to $\mu$, which is symbolically
shown by $\mu\circ\va^{-1}\ll\mu$. In this case, \emph{Radon-Nikodym
property} is denoted by $h:=\frac{d\mu\circ\varphi^{-1}}{d\mu}$.

A $\Sigma$-measurable transformation $\varphi: X\rightarrow X$ is
called \emph{periodic} if $\varphi^m=I$  for some $m\in \mathbb{N}$.
 It is called \emph{aperiodic}, if it
is not periodic. Also, if for each subset $F\in \Sigma$ with finite
measure, there exists an $N\in \mathbb{N}$ such that $F \cap
\varphi^n(F)=\emptyset$ for every $n>N$, then $\va$ is called
\emph{finitely non-mixing}.

 Set
$\Sigma_\infty=\bigcap_{n=1}^{\infty}\va^{-n}(\Sigma)$ and suppose
that $h$ is $\Sigma_\infty$-measurable. The assumption
$\mu\circ\va^{-1}\ll\mu$ implies that $\mu\circ\va^{-n}\ll\mu$ for
all $n\in\mathbb{N}$ and then
\begin{align*}
h_n:&=\frac{d\mu\circ\varphi^{-n}}{d\mu}=\frac{d\mu\circ\varphi^{-n}}
{d\mu\circ\varphi^{-(n-1)}}\cdots\frac{d\mu\circ\varphi^{-1}}{d\mu}\\&
=(h\circ\varphi^{-(n-1)})\cdots
(h\circ\varphi^0)=\prod_{i=0}^{n-1}h\circ\varphi^{-i}.
\end{align*}
Note that always $h\circ\va>0$ and $h_n=h^n$ whenever $h\circ\va=h$.
When it is restricted to a $\sigma$-subalgebra $\mathcal{A}$, is
denoted by
  $h_n^{\mathcal{A}}=\frac{d(\mu\circ
\varphi^{-n}|_{\mathcal{A}})}{d(\mu|_{\mathcal{A}})}$.

 The
\emph{change of variable formula}
$$\int_{\va^{-n}(A)}f\circ\va^n d\mu=\int_A h_nfd\mu, \ \ \ \ A\in\Sigma,
 \ f\in L^1(\Sigma),$$
 will be used frequently.

When $\va(\Sigma)\subseteq\Sigma$ and $\mu\circ\va\ll\mu$, then a
measure $\mu$ is called  \emph{normal} with respect to $\va$ and
 in this case $h^\sharp=\frac{d\mu\circ\varphi}{d\mu}$ is defined.
  Now, consider that
$$h^\sharp=(\frac{d\mu}{d\mu\circ\varphi})^{-1}=(\frac{d\mu\circ\varphi^{-1}}{d\mu}
\circ\va)^{-1}=\frac{1}{h\circ\va}$$ and
\begin{align*}
h_n^\sharp:&=\frac{d\mu\circ\varphi^{n}}{d\mu}=(h^\sharp\circ\varphi^{(n-1)})\cdots
(h^\sharp\circ\varphi^0)=\prod_{i=0}^{n-1}h^\sharp\circ\varphi^{i}
=\prod_{i=1}^{n}(h\circ\varphi^{i})^{-1},
\end{align*}
$h^\sharp_n\circ\varphi>0$, $h^\sharp_{n+1}=h^\sharp h^\sharp_n\circ
 \varphi$.\\
Let $1\leq p\leq\infty$. For any non-negative $\Sigma$-measurable
functions $f$ or for any  $f\in L^p(\Sigma)$,  Radon-Nikodym
Theorem, ensures the existence of  a unique ${\mathcal
A}$-measurable function $E^{\mathcal{A}}(f)$ such that
$$\int_{A}E^{\mathcal{A}}(f)d\mu=\int_{A}fd\mu, \ \  \ \ \ \mbox{for all} \ \
A\in {\mathcal A}.$$ A contractive projection $E^{\mathcal{A}}:
L^p(\Sigma)\rightarrow L^p({\mathcal A})$  is called a
\emph{conditional expectation operator} associated with the
$\sigma$-finite subalgebra ${\mathcal A}$.

Here, we list  some useful properties of the conditional expectation
operator:
\begin{itemize}
\item $E^{\mathcal{A}}(1)=1$;
  \item If $g$ is $\mathcal{A}$-measurable, then
 $E^{\mathcal{A}}(fg)=E^{\mathcal{A}}(f)g$;
  \item $|E^{\mathcal{A}}(f)|^p\le
  E^{\mathcal{A}}(|f|^p)$;
  \item For each $f\ge 0, \sigma(f)\subseteq
  \sigma(E^{\mathcal{A}}(f))$;
  \item Monotonicity: If $f$ and $g$ are real-valued with $f\le g$,
 then $E^{\mathcal{A}}f\le E^{\mathcal{A}}g$;
  \item For each $f\ge 0, E^{\mathcal{A}}(f)\ge 0$.
  \item $h_{n+1}=hE^{\va^{-1}(\Sigma)}(h_n)\circ
\varphi^{-1}=h_nE^{\va^{-n}(\Sigma)}(h)\circ\varphi^{-1}$\cite{hoo}.
\end{itemize}
A detailed information of the condition expectation operator may be
found in \cite{rao, rao2, her, l}.

A \emph{weighted composition operator} $uC_\va: L^p(\Sigma)
 \rightarrow L^p(\Sigma)$ defined by $f \mapsto u f\circ \va$
is bounded if and only if $J\in L^\infty(\Sigma)$, where
$J:=hE^{\mathcal{A}}(|u|^{p})\circ \varphi^{-1}$, and in this case
$\|uC_\va\|^p=\|J\|_\infty$ (see \cite{hoo, sin, azz}).

Now, we are ready to define a
$\emph{conditional weighted composition operator}$ $T_u$ by:\\
$$T_u : L^{p}(\Sigma)\rightarrow
L^p(\mathcal{A})$$
$$T_uf:=E^{\mathcal{A}}\circ uC_\va (f)= E^{\mathcal{A}}(u f \circ \va).$$
For the fundamental properties of the conditional type operators,
the reader is refereed to  \cite{pag1, pag2, es1, es2}.

 The hypercyclicity of the well-known operators such
as weighted shifts, weighted translations, conditional weighted
translations and weighted composition operators in different
settings has been studied in \cite{sal, chen, dar, az, az2, az3,
aba, be, yus}.

Separability and infinite-dimension are two essential objects for
the underlying space to admit a hypercyclic vector \cite{bay, mang}.
To that end, it is important to know  that $L^{p}(X,\Sigma,\mu)$ is
separable if and only if $(X, \Sigma, \mu)$ is separable, i.e.,
there exists a countable $\sigma$-subalgebra $\mathcal{F}\subseteq
\Sigma$  such that for each $\epsilon>0$ and $A\in \Sigma$ we have
$\mu(A\Delta B)<\epsilon$ for some $B\in \mathcal{F}$.  For more
details consult \cite{rao2}.

 In this paper, we will survey the dynamics of a
conditional weighted composition operator
$T_u=E^{\mathcal{A}}(uf\circ\va)$ on $L^p(\Sigma)$ spaces. First, we
prove that $T_u$  cannot  be $L^p(\mathcal{A})$-hypercyclic if
$\varphi$ is a periodic non-singular transformation. In addition,
the necessary conditions for the subspace-hypercyclicity of $T_u$
are then given provided that $\varphi$ is non-singular and finitely
non-mixing. For the sufficient conditions, we also require that
$\varphi$ is normal.
 The subspace-weakly mixing
and subspace-topologically mixing concepts are also studied for
$T_u$. At the end, about what we argued, an examples is given.

\section{\textbf{Subspace-hypercyclicity of $T_u$
 On $L^p(\Sigma)$ }}

 In this section,  the $L^p(\mathcal{A})$-hypercyclicity
  of a conditional weighted
composition operator $T_u$ is studied. When
 $\varphi$ is periodic transformation, it is seen that $T_u$ is not
 $L^p(\mathcal{A})$-hypercyclic. But, when it is aperiodic, by  Kitai's
 subspace-hypercyclicity criterion
 we obtain some necessary and then sufficient conditions for $T_u$ to be
 subspace-hypercyclic.  We are thankful to the techniques used
in \cite{chen, sal}.

\begin{thm}
Let $\varphi$ be a periodic \emph{non-singular} transformation and
$\va^{-1}\mathcal{A}\subseteq \mathcal{A}$. Then a conditional
weighted composition operator $T_u: L^{p}(\Sigma)\rightarrow
L^p(\mathcal{A})$ is not subspace-hypercyclic with respect to
$L^p(\mathcal{A})$, for each $1\le p<\infty$.
\end{thm}
\begin{proof}
Suppose that there exists an $m\in \mathbb{N}$ such that
$\varphi^m=I$. Since $\va^{-1}\mathcal{A}\subseteq \mathcal{A}$, the
orbit of $T_u$ at each $f\in L^{p}(\Sigma)$ is written as follows:
\begin{align*}
& orb(T_u, f)=\{f, T_u f, \cdots , T_u^{m}f\}\cup\{T_u^{m+1}f, T_u^{m+2}f,
\cdots , T_u^{2m}f\} \cup\cdots \\
& \cup\{T_u^{km+1}f, T_u^{km+2}f, \cdots ,
T_u^{(k+1)m}f\}\cup\cdots\\
& =\big\{f, E^{\mathcal{A}}(u f\circ \varphi),
E^{\mathcal{A}}(u)E^{\mathcal{A}}(u f\circ \varphi)\circ
\varphi,\cdots, \prod_{i=0}^{m-2}E^{\mathcal{A}}(u)\circ\va^{i}
E^{\mathcal{A}}(u f\circ \va)\circ\va^{m-1}\big\}\\
&\cup \big\{\prod_{i=0}^{m-1}E^{\mathcal{A}}(u)\circ\va^{i} E^{\mathcal{A}}(u f\circ
\varphi), \prod_{i=0}^{m-1}E^{\mathcal{A}}(u)\circ\va^{i} E^{\mathcal{A}}(u)E^{\mathcal{A}}(u
f\circ
\varphi)\circ \varphi, \cdots,\\
&  \prod_{i=0}^{m-1}E^{\mathcal{A}}(u)\circ\va^{i}
\prod_{i=0}^{m-2}E^{\mathcal{A}}(u)\circ\va^{i}E^{\mathcal{A}}(u f\circ
\va)\circ\va^{m-1}\big\}\\
&\cup \big\{(\prod_{i=0}^{m-1}E^{\mathcal{A}}(u)\circ\va^{i})^2 E^{\mathcal{A}}(u f\circ
\varphi), (\prod_{i=0}^{m-1}E^{\mathcal{A}}(u)\circ\va^{i})^2
E^{\mathcal{A}}(u)E^{\mathcal{A}}(u f\circ
\varphi)\circ \varphi, \cdots,\\
&  (\prod_{i=0}^{m-1}E^{\mathcal{A}}(u)\circ\va^{i})^2
\prod_{i=0}^{m-2}E^{\mathcal{A}}(u)\circ\va^{i}E^{\mathcal{A}}(u f\circ
\va)\circ\va^{m-1}   \big\} \cup\\
& \vdots
\end{align*}
 Now we consider that  $\|\prod_{i=0}^{m-1}E^{\mathcal{A}}(u)
 \circ\va^{i}\|_\infty\leq 1$. Since
$\|T_u\|\leq \|J\|_{\infty}^{1/p}$, $\|T_u^n\| \leq\|T_u\|^n \leq
\|J\|_\infty^{n/p}$, and for each $n\in\mathbb{N}$ we have
\begin{align*}
\|T_u^n f\|_p &\leq \max\{\|f\|_p, \|E^{\mathcal{A}}(u f\circ
\va)\|_p, \|E^{\mathcal{A}}(u) E^{\mathcal{A}}(u f\circ
\va)\circ\va\|_p,\cdots,\\
& \|\prod_{i=0}^{m-2}E^{\mathcal{A}}(u)\circ\va^{i} E^{\mathcal{A}}(u f\circ
\va)\circ\va^{m-1}\|_p\}\\& \leq \|f\|_p \max\{1, \|J\|_\infty
^{\frac{1}{p}}, \|J\|_{\infty}^{\frac{2}{p}}, \cdots,
\|J\|_\infty^{\frac{m-1}{p}} \}.
\end{align*}
 Therefore, $orb(T_u, f)$ is a bounded subset
  and  cannot be dense in $L^p(\mathcal{A})$.

  Iv the second case
  $\|\prod_{i=0}^{m-1}E^{\mathcal{A}}(u)\circ\va^{i}\|_{\infty}>1$,
 assume that $T_u$ is subspace-hypercyclic with respect to $L^p(\mathcal{A})$.
Then there exists a  subset $F\in \mathcal{A}$ with $0<\mu(F)<
\infty$ for each $\varepsilon>0$, such that
$|\prod_{i=0}^{m-1}E^{\mathcal{A}}(u)\circ\va^{i}|>1$. Then there is
a subspace-hypercyclic vector $f\in L^{p}(\mathcal{A})$ and $n\in
\mathbb{N}$ such that
$$\|f- 2 \chi_F\|_p < \varepsilon \quad \mbox{and} \quad
\|(T_u^{m+1})^n f\|_p<\varepsilon.$$  We set $S=\{t\in F: |f(t)|<1
\}$ and note that  $\chi_S\leq \chi_S|f-2|\leq\chi_S|f-2\chi_F|$.
Thus,  $\mu(S)< \varepsilon^p$. On the other hand,
\begin{align*}
  \varepsilon^p>\|(T_u^{m})^n
f)\|_p^p&= \int_X |\prod_{i=0}^{mn-1}E^{\mathcal{A}}(u)\circ\va^{i}f\circ\va^{mn}|^p d\mu\\
&= \int_X |\prod_{i=0}^{m-1}E^{\mathcal{A}}(u)\circ\va^{i}|^{np}|f|^p d\mu
  \geq \int_{F-S} | f|^p d\mu
  \geq \mu(\chi_{F-S}).
\end{align*}
Therefore, $\mu(F)=\mu(S)+ \mu(F-S)<2\varepsilon^p$, which is a
contradiction.
\end{proof}
\begin{rem}
If $\varphi$ is a periodic \emph{non-singular} transformation,
$\va^{-1}\mathcal{A}\subseteq \mathcal{A}$ and $u=1$, then a
conditional  composition operator $T_uf=E^{\mathcal{A}}(f\circ \va)$
is not subspace-hypercyclic with respect to $L^p(\mathcal{A})$
either. Since its orbit at $f\in L^{p}(\Sigma)$ i.e.,
$orb(T_u,f)=\{f, E^{\mathcal{A}}(f\circ \varphi),
E^{\mathcal{A}}(f\circ \varphi)\circ \va, E^{\mathcal{A}}(f\circ
\varphi)\circ \va^2 \cdots, E^{\mathcal{A}}(f\circ
\varphi)\circ\va^{m-1}\}$ is a bounded subset. Indeed,
$$\|T_u^n f\|_p \leq \|f\|_p \max\{1, \|h\|_\infty ^{\frac{1}{p}},
\|h\|_{\infty}^{\frac{2}{p}}, \cdots, \|h\|_\infty^{\frac{m-1}{p}}
\}.$$
\end{rem}
\begin{cor}
Suppose that $\mathcal{A}=\va^{-1}\Sigma$ and $\varphi$ is a
periodic \emph{non-singular} transformation. Then $$orb(T_u,
f)=\big\{f, E^{\va^{-1}\Sigma}(u)f\circ \varphi,
E^{\va^{-1}\Sigma}(u)E^{\va^{-1}\Sigma}(u)\circ \varphi f\circ
\varphi^2,\cdots, \prod_{i=0}^{m-1}E^{\va^{-1}\Sigma}(u)\circ\va^{i} f\big\}$$ and hence
$T_u$ is not subspace-hypercyclic with respect to
$L^p(\va^{-1}\Sigma)$, for each $1\le p<\infty$.
\end{cor}
\begin{thm}\label{T1}
Let $\varphi: X\rightarrow X$ be  a non-singular and finitely
non-mixing transformation and $\va^{-1}\mathcal{A}\subseteq
\mathcal{A}$. Suppose that  $T_u: L^{p}(\Sigma)\rightarrow
L^p(\mathcal{A})$ is subspace-hypercyclic with respect to
$L^p(\mathcal{A})$. Then
 for each  subset $F\in \mathcal{A}$ with
  $0<\mu(F)< \infty$,
there exists a sequence of $\mathcal{A}$-measurable sets
$\{V_k\}\subseteq F$ such that $\mu(V_k)\rightarrow \mu(F)$ as
$k\rightarrow\infty$, and there is a sequence of integers $(n_k)$
such that $$\lim_{k\rightarrow
\infty}\|(\prod_{i=0}^{{n_k}-1}E^{\mathcal{A}}(u)\circ\va^{i})^{-1}|_{V_k}\|_{\infty}
=0$$
  and $$\lim_{k\rightarrow
\infty}\|\sqrt[p]{h^{\mathcal{A}}_{n_k}}
[E^{\va^{-n_k}(\mathcal{A})}(\prod_{i=0}^{{n_k}-1}
  E^{\mathcal{A}}(u)\circ\va^{i})]\circ \va^{-n_k}|_
  {V_k}\|_{\infty}=0.$$
\end{thm}
\begin{proof}

 Let $F\in \mathcal{A}$  be an arbitrary
 set with $0<\mu(F)< \infty$ and let $\varepsilon>0$ be an arbitrary.
 A transformation $\varphi$ is   finitely
non-mixing and hence, there is  an $N\in \mathbb{N}$ such that
$F\cap \varphi^n(F)=\emptyset$ for each $n> N$. Choose
$\varepsilon_1$ such that
$0<\varepsilon_1<\frac{\varepsilon}{1+\varepsilon}$. Since the set
of all subspace-hypercyclic vectors for $T_u$, is dense  in
$L^p(\mathcal{A})$, there exist a subspace-hypercyclic vector $f\in
L^{p}(\mathcal{A})$ and $m\in \mathbb{N}$ with $m>N$ such that
$$\|f- \chi_F\|_p<\varepsilon_1^2 \quad \mbox{and} \quad
\|T_u^{m}f - \chi_F\|_p<\varepsilon_1^2.$$   Put
$P_{\varepsilon_1}=\{t\in F: |f(t)-1|\geq \varepsilon_1 \}$ and
$R_{\varepsilon_1}=\{t\in X-F: |f(t)|\geq \varepsilon_1\}$. Then we
have
\begin{eqnarray*}
  \varepsilon_1^{2p} &>& \|f- \chi_F\|_p^p=\int_X |f- \chi_F|^p d\mu \\
  &\geq& \int_{P_{\varepsilon_1}} |f(x)- 1|^p d\mu(x)+\int_{R_{\varepsilon_1}}|f(x)|^p d\mu(x) \\
   &\geq& \varepsilon_1^p(\mu(P_{\varepsilon_1})+\mu(R_{\varepsilon_1})).
\end{eqnarray*}
Then, $\max\{\mu(P_{\varepsilon_1}),
\mu(R_{\varepsilon_1})\}<\varepsilon_1^p$.
 Set $S_{m, \varepsilon_1}=\{t\in F:
|\prod_{i=0}^{m-1}E^{\mathcal{A}}(u)\circ\va^{i}
f\circ\va^{m}(t)-1|\geq \varepsilon_1\}$ and now consider the
following relationships:
\begin{eqnarray*}
  \varepsilon_1^{2p} &>& \|T_u^{m}f - \chi_F\|_p^p\\
      &=& \int_{X} | \prod_{i=0}^{m-2}E^{\mathcal{A}}(u)\circ\va^{i}E^{\mathcal{A}}(u f\circ
\va)\circ\va^{m-1} - \chi_F|^p d\mu\\
   &\geq& \int_{S_{m, \varepsilon_1}}
   |\prod_{i=0}^{m-2}E^{\mathcal{A}}(u)\circ\va^{i}E^{\mathcal{A}}(u f\circ
\va)\circ\va^{m-1}(t)-1|^p d\mu(t)\\
&\geq& \int_{S_{m, \varepsilon_1}}
   |\prod_{i=0}^{m-1}E^{\mathcal{A}}(u)\circ\va^{i} f\circ\va^{m}(t)-1|^p d\mu(t)\\
&\geq& \varepsilon_1^p \mu(S_{m, \varepsilon_1})
  \end{eqnarray*}
  to deduce that $\mu(S_{m,
  \varepsilon_1})<\varepsilon_1^p$. But for an arbitrary $t\in F$,
  it is readily seen
  that $\varphi^{m}(t) \notin F$ because of $F\cap \varphi^{-m}(F)=\emptyset$.
  Hence, for each $t\in F-(S_{m, \varepsilon_1}\cup
   \varphi^{-m}(R_{\varepsilon_1}))$, we have
  $$|(\prod_{i=0}^{m-1}E^{\mathcal{A}}(u)\circ\va^{i})^{-1}(t)|<\frac{|f\circ
   \varphi^{m}(t)|}{1-\varepsilon_1}<
  \frac{\varepsilon_1}{1-\varepsilon_1}<\varepsilon.$$
 Now, let $U_{m,\varepsilon_1}=\va^{-m}(\{t\in F: \sqrt[p]{h^{\mathcal{A}}_{m}(t)}\ |
 E^{\va^{-m}(\mathcal{A})}(\prod_{i=0}^{m-1}E^{\mathcal{A}}(u)\circ\va^{i})\circ
  \va^{-m}(t)f(t)|\geq \varepsilon_1\})$. Here, we remind that $\prod_{i=0}^{m-1}E^{\mathcal{A}}(u)\circ\va^{i}
       \circ\varphi^{-m}=\prod_{i=1}^{m}E^{\mathcal{A}}(u)\circ\va^{-i}$ on
       $\sigma(h_m^{\mathcal{A}})$.
 Use the change of variable formula  to obtain that
\begin{eqnarray*}
  \varepsilon_1^{2p} &>& \|T_u^{m}f - \chi_F\|_p^p\\
      &=& \int_{X} |\prod_{i=0}^{m-1}E^{\mathcal{A}}(u)\circ\va^{i}
       f\circ\va^{m} - \chi_F|^p d\mu\\
      &\geq& \int_{X} |E^{\va^{-m}(\mathcal{A})}(\prod_{i=0}^{m-1}
      E^{\mathcal{A}}(u)\circ\va^{i}) f\circ \varphi^{m}
       - E^{\va^{-m}(\mathcal{A})}(\chi_F)|^p d\mu\\
       &\geq& \int_{U_{m, \varepsilon_1}} |E^{\va^{-m}(\mathcal{A})}
       (\prod_{i=0}^{m-1}E^{\mathcal{A}}(u)\circ\va^{i})
        f\circ   \varphi^{m} |^p d\mu\\
       &\geq& \int_{\va^m(U_{m, \varepsilon_1})}
       |E^{\va^{-m}(\mathcal{A})}(\prod_{i=0}^{m-1}E^{\mathcal{A}}(u)\circ\va^{i})
       \circ\varphi^{-m} f |^p h_m^{\mathcal{A}}d\mu\\
   &\geq& \varepsilon_1^p \mu(\va^m(U_{m, \varepsilon_1})),
  \end{eqnarray*}
  which implies in turn  that $\mu(\va^m(U_{m,
  \varepsilon_1}))<\varepsilon_1^p$. That
  $E^{\va^{-m}(\mathcal{A})}(\chi_F)=0$ is concluded of  the fact that
  $F\cap \varphi^{-m}(F)=\emptyset$.
  Note that for each $t\in F-(\va^m(U_{m,
  \varepsilon_1}) \cup P_{\varepsilon_1})$, we have
$$\sqrt[p]{h_{m}(t)}\ | E^{\va^{-m}(\mathcal{A})}(\prod_{i=0}^{m-1}E^{\mathcal{A}}(u)\circ\va^{i})\circ \va^{-m}
(t)f(t)|<\frac{\varepsilon_1}{1-\varepsilon_1}<\varepsilon.$$
Finally, put $V_{m, \varepsilon_1}:= F-(P_{\varepsilon_1}\cup
\varphi^{-m}(R_{m, \varepsilon_1}) \cup S_{m, \varepsilon_1}\cup
\va^m(U_{m, \varepsilon_1}))$. Then, clearly  $\mu(F- V_{m,
  \varepsilon_1})<4\varepsilon_1^p$,
  $\|(\prod_{i=0}^{m-1}E^{\mathcal{A}}(u)\circ\va^{i})^{-1}|
  _{V_{m, \varepsilon_1}}\|_{\infty}<\varepsilon$
  and
  $$\|\sqrt[p]{h_{m}}[E^{\va^{-m}(\mathcal{A})}(\prod_{i=0}^{m-1}E^{\mathcal{A}}(u)
  \circ\va^{i})]\circ \va^{-m}|_
  {V_{m, \varepsilon_1}}\|_{\infty}<\varepsilon.$$
By induction, for each $k\in \mathbb{N}$ we get a measurable subset
$V_k\subseteq F$ and an increasing subsequence $(n_k)$ such that
$\mu(F- V_{k})<4(\frac{1}{k})^p$,
$\|(\prod_{i=0}^{n_k-1}E^{\mathcal{A}}(u)\circ\va^{i})^{-1}|_{V_k}\|_{\infty}<\varepsilon$
  and $\|\sqrt[p]{h^{\mathcal{A}}_{n_k}}[E^{\va^{-n_k}(\mathcal{A})}(\prod_{i=0}^{n_k-1}
  E^{\mathcal{A}}(u)\circ\va^{i})]\circ \va^{-n_k}|_
  {V_k}\|_{\infty}<\varepsilon$.
\end{proof}

\begin{thm}\label{T11}
Let $T_u: L^{p}(\Sigma)\rightarrow L^p(\mathcal{A})$ be bounded
with $\sigma(u)=X$, and let $\varphi$ be
  a normal and finitely non-mixing
transformation provided that $\va^{-1}\mathcal{A} \subseteq
\mathcal{A}\subseteq \Sigma_\infty$ and
$\sup_n\|h_n^{\mathcal{A}\sharp}\|_{\infty}<\infty$. If for each
subset $F\in \mathcal{A}$ with $0<\mu(F)< \infty$, there exists a
sequence of $\mathcal{A}$-measurable sets $\{V_k\}\subseteq F$ such
that $\mu(V_k)\rightarrow \mu(F)$ as $k\rightarrow\infty$, and there
is a sequence of integers $(n_k)$ such that
$$\lim_{k\rightarrow \infty}\|(\prod_{i=0}^{n_k-1}E^{\mathcal{A}}(u)
\circ\va^{i})^{-1}|_{V_k}\|_{\infty}=0$$
  and $$\lim_{k\rightarrow \infty}\|\sqrt[p]{h_{n_k}}\ [\prod_{i=0}^{{n_k}-1}
  E^{\mathcal{A}}(u)\circ\va^{i}]\circ
   \va^{-n_k}|_{V_k}\|_
  {\infty}=0,$$ then $T_u$ is
  subspace-hypercyclic with respect to
$L^p(\mathcal{A})$.
\end{thm}

\begin{proof}
 Since, $S^{\mathcal{A}}(X)$ is
dense in $L^p(\mathcal{A})$,  we may take
$D_1=D_2=S^{\mathcal{A}}(X)$ in the subspace-hypercyclicity's
criterion. For an arbitrary $f\in S^{\mathcal{A}}(X)$, one can
easily find $\{V_k\}\subseteq \sigma(f)$ such that
$\mu(V_k)\rightarrow \mu(\sigma(f))$ and finds an $N_1$ such that
$\sigma(f)\cap \varphi^n( \sigma(f))=\emptyset$ for each $n>N_1$.
 Now, for each $n_k>N_1$ define the vector
$f_{k}= \frac{ f\circ
\varphi^{-n_k}}{[\prod_{i=0}^{n_k-1}E^{\mathcal{A}}(u)\circ\va^{i}]\circ
\varphi^{-n_k}}$. Since $\va^{-1}\mathcal{A}\subseteq
\mathcal{A}\subseteq \Sigma_\infty$, then $f_k\in L^p(\mathcal{A})$
and the simple computations show that $T_u^{n_k}f_k=f$.
Now, we will show that   $\|T_u^{n_k} f\|_p \rightarrow 0$ and
$\|f_k \|_p \rightarrow 0$ as $k\rightarrow \infty$. For an
arbitrary $\varepsilon>0$,  there exist $M, N_1\in \mathbb{N}$,
sufficiently large such that $ V_{N_1}\subseteq\sigma(f)$ and
$$\mu(\sigma(f)-V_{N_1})<\frac{\varepsilon}{M \|f\|_{_\infty}^p}.$$ By
 Egoroff's theorem, there exists an $N_2$ such that for each $n_k>
 N_2$, \break
$\|\sqrt[p]{h^{\mathcal{A}}_{n_k}}
[\prod_{i=0}^{n_k-1}E^{\mathcal{A}}(u)\circ\va^{i}]\circ\va^{-n_k}\|^p_{\infty}
<\frac{\varepsilon}{\|f\|^p_{\infty}}$ on $V_{N_1}$. So, there
exists a non-negative real number $M$ such that
$\|\sqrt[p]{h^{\mathcal{A}}_{n_k}}[\prod_{i=0}^{n_k-1}E^{\mathcal{A}}(u)\circ\va^{i}]
\circ\va^{-n_k}\|_{\infty}^p\leq M<\infty$  on $\sigma(f)$.  Now, by
the change of variable formula, for each $n_k> N=\max\{N_1, N_2\}$
we have
\begin{align*}
  \|T_u^{n_k} f\|_p^p&=
   \int_X|\prod_{i=0}^{n_k-2}E^{\mathcal{A}}(u)\circ\va^{i}E^{\mathcal{A}}(u f\circ
\va)\circ\va^{n_k-1}|^pd\mu \\
&=\int_X|\prod_{i=0}^{n_k-1}E^{\mathcal{A}}(u)\circ\va^{i} f\circ\va^{n_k}|^pd\mu \\
&=\int_{\sigma(f)} |\prod_{i=0}^{n_k-1}E^{\mathcal{A}}(u)\circ\va^{i}\circ\va^{-n_k}f|^ph_{n_k} d\mu\\
     &= \int_{{\sigma(f)-V_N}}
    |\prod_{i=0}^{n_k-1}E^{\mathcal{A}}(u)\circ\va^{i}\circ\va^{-n_k}f|^p h_{n_k} d\mu\\
    &+ \int_{V_N}
    |\prod_{i=0}^{n_k-1}E^{\mathcal{A}}(u)\circ\va^{i}\circ\va^{-n_k}f|^p h_{n_k} d\mu\\
   &< \|\sqrt[p]{h_{n_k}} \prod_{i=0}^{n_k-1}E^{\mathcal{A}}(u)
   \circ\va^{i}\circ\va^{-n_k}\|_{\infty}^p \|f\|_{\infty}^p \ \mu(\sigma(f)-V_N)
   \\&+ \frac{\varepsilon}{\|f\|_{\infty}^p} \|f\|_{\infty}^p
<2\varepsilon.
\end{align*}
By taking into account that
$\sup_n\|h_n^{\mathcal{A}\sharp}\|_{\infty}<\infty$, we have
\begin{align*}
  \lim_{k\rightarrow \infty}\|f_k\|_p^p&=\lim_{k\rightarrow \infty}
   \int_X|\frac{f\circ\va^{-n_k}}
   {\prod_{i=0}^{n_k-1}E^{\mathcal{A}}(u)\circ\va^{i}\circ \varphi^{-n_k}}|^pd\mu \\
&=\lim_{k\rightarrow \infty}\int_{\sigma(f)}
|\frac{f}{\prod_{i=0}^{n_k-1}E^{\mathcal{A}}(u)\circ\va^{i}}|^p
h_{n_k}^\sharp d\mu\\
&\leq\sup_{k}
\|h_{n_k}^{\mathcal{A}\sharp}\|_{_\infty}(\lim_{k\rightarrow
\infty}\int_{\sigma(f)-V_N}
    |\frac{f}{\prod_{i=0}^{n_k-1}E^{\mathcal{A}}(u)\circ\va^{i}}|^p
    d\mu\\
 &+ \lim_{k\rightarrow \infty}\int_{V_N}
    |\frac{f}{\prod_{i=0}^{n_k-1}E^{\mathcal{A}}(u)\circ\va^{i}}|^p  d\mu)\\
    &=0.
\end{align*}
 Finally,  it is clear that $T_u^{n_k}L^p(\mathcal{A})\subseteq L^p(\mathcal{A})$
  for all $k\in\mathbb{N}$, because of $\va^{-1}\mathcal{A}
  \subseteq \mathcal{A}$ and hence  $T_u$ satisfies in the subspace-hypercyclicity
criterion  and  is subspace-hypercyclic.
\end{proof}
\begin{prop}\label{P1}
Suppose that $\varphi: X\rightarrow X$ is  a normal and finitely
non-mixing transformation with $\va^{-1}(\mathcal{A})\subseteq
\mathcal{A}\subseteq \Sigma_{\infty}$.
 Let $\sup_n\|h_n^\sharp\|_{\infty}<\infty$ and $\sigma(u)=X$.
 Then the following conditions are equivalent:
\begin{itemize}
  \item [(i)] $T_u$ satisfies the subspace-hypercyclic criterion.
  \item [(ii)] $T_u$ is  subspace-hypercyclic with respect to
$L^p(\mathcal{A})$.
\item [(iii)] $T_u \oplus T_u$ is subspace-hypercyclic with respect to
 $L^{p}(\mathcal{A})\oplus L^{p}(\mathcal{A})$.
\item [(iv)] $T_u$ is subspace-weakly mixing.
\end{itemize}
\end{prop}
\begin{proof}
$(i)\Rightarrow (ii)$.  Note that if an operator satisfies the
subspace-hypercyclic criterion, then it is subspace-transitive and
hence is subspace-hypercyclic \cite[Theorem 3.5]{mad}. For the
implication $(ii)\Rightarrow (iii)$, we show that $T_u \oplus T_u$
is subspace-topologically transitive, according \cite[Theorem
3.3]{mad}. To begin, pick two pairs of non-empty open sets $(A_1,
B_1)$ and $(A_2, B_2)$ in $L^p(\mathcal{A})\oplus
L^{p}(\mathcal{A})$ arbitrarily. For $j=1,2$, choose the functions
$f_j, g_j\in S^{\mathcal{A}}(X)$ with $f_j\in A_j$ and $g_j\in B_j$.
Let $F=\sigma(f_1)\cup \sigma(f_2)\cup \sigma(g_1)\cup \sigma(g_2)$.
Then $\mu(F)<\infty$. Assume that $\{V_k\}\subseteq F$,
$\{(\prod_{i=0}^{n_k-1}E^{\mathcal{A}}(u)\circ\va^{i})^{-1}\}$ and
$\{\sqrt[p]{h^{\mathcal{A}}_{n_k}}\
E^{\va^{-n_k}(\mathcal{A})}(\prod_{i=0}^{n_k-1}E^{\mathcal{A}}(u)\circ\va^{i})\circ
\va^{-n_k}\}$ are as provided by Theorem \ref{T1}. There is an
$N_1\in \mathbb{N}$, such that for all $n>N_1$, $F\cap
\varphi^n(F)=\emptyset$. Moreover, for each $\varepsilon>0$ there
exists $N_2\in \mathbb{N}$, such that for each $k>N_2$ and $n_k
>N_1$,  $\|\sqrt[p]{h^{\mathcal{A}}_{n_k}}\ E^{\va^{-n_k}(\mathcal{A})}(\prod_{i=0}^{n_k-1}
E^{\mathcal{A}}(u)\circ\va^{i})\circ \va^{-n_k}|_
  {V_k}\|_{\infty}^p< \frac{\varepsilon}{\|f_j\|_p^p}$ on
$V_k$. Hence, for $k>N_2$, we get that
\begin{align*}
\|T_u^{n_k}(f_j \chi_{V_k})\|_p^p&=
\int_X|T_u^{n_k} (f_j \chi_{V_k})|^pd\mu \\
&=\int_{X} |\prod_{i=0}^{n_k-1}E^{\mathcal{A}}(u)\circ\va^{i} (f_j\chi_{V_k})\circ
\varphi^{n_k})|^p \ d\mu\\
&= \int_{V_k}
|[\prod_{i=0}^{n_k-1}E^{\mathcal{A}}(u)\circ\va^{i}]\circ\va^{-n_k}
f_j|^p h_{n_k} d\mu <\varepsilon.
\end{align*}
Now, define a map $D_{\varphi}(f)= \frac{f\circ
\varphi^{-1}}{E^{\mathcal{A}}(u)\circ \varphi^{-1}}$ on the subspace
$S^{\mathcal{A}}(X)$. Then for each $f\in S^{\mathcal{A}}(X)$,
$T_u^{n_k}D_{\varphi}^{n_k}(f)= f$. Again, we may find an $N_3\in
\mathbb{N}$ such that for each $k>N_3$ and $n_k
>N_1$,  $\|(\prod_{i=0}^{n_k-1}E^{\mathcal{A}}(u)\circ\va^{i})^{-1}\|_{\infty}^p<
 \frac{\varepsilon}{M\|g_j\|_\infty^p}$ on
$V_k$, where $M=\sup_n\|h_n^{\mathcal{A}\sharp}\|_{\infty}<\infty$.
On the other hand, for each $k>N_3$ note that
\begin{align*}\label{E2}
\|D_{\varphi}^{n_k}(g_j\chi_{V_k})\|_p^p&=
\int_{\va^{n_k}(V_k)}|\frac{g_j\circ\va^{-n_k}}{[\prod_{i=0}^{n_k-1}E^{\mathcal{A}}(u)
\circ\va^{i}]\circ\va^{-n_k}}|^pd\mu \\
&=\int_{V_k}|\frac{g_j}{\prod_{i=0}^{n_k-1}E^{\mathcal{A}}(u)\circ\va^{i}}|^p\ h_n^\sharp
d\mu <\varepsilon.
\end{align*}
For each $k\in \mathbb{N}$, let $f^\natural_{j,k}= f_j \chi_{V_k}+
D_{\varphi}^{n_k}(g_j\chi_{V_k})$. Then we have $f^\natural_{j,k}\in
L^{p}(\mathcal{A})$,
$$\|f^\natural_{j,k}- f_j\|_p^p\leq \|f_j\|_\infty^p\ \mu(F-V_k)+
 \|D_{\varphi}^{n_k}(g_j\chi_{V_k})\|_p^p$$
and
$$\|T_u^{n_k}f^\natural_{j,k}- g_j\|_p^p\leq \|g_j\|_\infty^p\
\mu(F-V_k)+
 \|T_u^{n_k}(f_j \chi_{V_k})\|_p^p.$$
Hence,  $\lim_{k\rightarrow \infty}f^\natural_{j,k}=f_j$,
$\lim_{k\rightarrow \infty}T_u^{n_k}f^\natural_{j,k}=g_j$ and
$T_u^{n_k}(A_j)\cap B_j\neq \emptyset$ for some $k\in \mathbb{N}$.
Moreover, since $\va^{-1}(\mathcal{A})\subseteq \mathcal{A}$ then
$T_u^{n_k}(L^{p}(\mathcal{A}))\subseteq L^{p}(\mathcal{A})$. So $T_u
\oplus T_u$ is subspace-hypercyclic on $L^{p}(\mathcal{A})\oplus
L^{p}(\mathcal{A})$.

 To prove the
implication $(iv)\Rightarrow (i)$, we use B$\grave{e}$s-Peris's
approach stated in \cite[Theorem 4.2]{bay}. Assume that $T_u \oplus
T_u$ is subspace-hypercyclic on $L^{p}(\mathcal{A})\oplus
L^{p}(\mathcal{A})$ with subspace-hypercyclic vector $f \oplus g$.
Note that for each $n\in \mathbb{N}$, the operator $I \oplus T_u^n$
has dense range and commutes with $T_u \oplus T_u$, therefore $orb(I
\oplus T_u^n, f \oplus g)=(I \oplus T_u^n)orb(T_u \oplus T_u, f
\oplus g)$. Eventually $f \oplus T_u^ng$ is subspace-hypercyclic
vector as well. We show that the subspace-hypercyclic criterion is
satisfied by $D_1=D_2=orb(T_u \oplus T_u, f \oplus g)$. Let $U$ be
an arbitrary open neighborhood of $0$ in $L^{p}(\mathcal{A})$.
Hence, one can find a sequence $(g_k)\subseteq U$ and an increasing
sequence of integers $(n_k)$ such that $T_u^{n_k}f\oplus
T_u^{n_k}g_k\rightarrow 0\oplus g$ and $g_k\rightarrow 0$. Clearly,
$T_u^{n_k}(L^{p}(\mathcal{A}))\subseteq L^{p}(\mathcal{A})$.
\end{proof}
\begin{cor}\label{C1}
Under the assumptions of Proposition \ref{P1},
the following conditions are equivalent:
\begin{itemize}
  \item [(i)] $T_u$ is subspace-topologically mixing on $L^p(\mathcal{A})$.
  \item [(ii)] For each $\mathcal{A}$-measurable subset $F\subseteq X$ with
   $0<\mu(F)<\infty$,
  there exists a sequence of $\mathcal{A}$-measurable sets
   $\{V_n\}\subseteq F$
  such that $\mu(V_n)\rightarrow \mu(F)$ as $n\rightarrow
  \infty$ and
     $\lim_{n\rightarrow \infty}\|(\prod_{i=0}^{n-1}E^{\mathcal{A}}(u)\circ\va^{i})^{-1}|
  _{V_n}\|_{\infty}=\lim_{n\rightarrow \infty}
  \|\sqrt[p]{h^{\mathcal{A}}_{n}}( \prod_{i=0}^{n-1}E^{\mathcal{A}}(u)\circ\va^{i}\circ \varphi^{-{n}})|_
  {V_n}\|_{\infty}=0$.
\end{itemize}
\end{cor}
\begin{proof}
By Theorem \ref{T11} and Proposition \ref{P1} the implication
$(ii)\Rightarrow(i)$ is established, just use the full sequences
instead of subsequences.  For the implication $(i)\Rightarrow(ii)$,
let $\varepsilon>0$ and $F\in \mathcal{A}$ with $0<\mu(F)<\infty$ be
arbitrary. Consider a non-empty and open subset $U=\{f\in
L^{p}(\mathcal{A}): \|f- \chi_F\|_p<\varepsilon\}$. Since $T_u$ is
subspace-topologically mixing and $\varphi$ is finitely non-mixing,
one may find $N\in \mathbb{N}$ such that for all $n>N$,
$T_u^{n}(U)\cap U\neq \emptyset$ and $F\cap \varphi^n(F)=\emptyset$.
Hence, for each $n>N$, we can choose a function $f_n \in U$ such
that $T_u^{n}f_n \in U$. Then $\|f_n- \chi_F\|_p<\varepsilon$ and
$\|T_u^{n}f_n- \chi_F\|_p<\varepsilon$. The rest of the proof can be
proceed like as Theorem \ref{T1}.
\end{proof}

\begin{exam}
Let $X=\mathbb{R}$ be the  real line with  Lebesgue measure $\mu$ on
the $\sigma$-algebra $\Sigma$ of all Lebesgue measurable subsets of
$\mathbb{R}$. Let $\mathcal{A}$ be the $\sigma$-subalgebra generated
by the symmetric intervals about the origin.
  For a
positive real number $t$ define the transformation $\varphi:
\mathbb{R}\rightarrow \mathbb{R}$ by  $\varphi(x)=x+t, ~~x\in
\mathbb{R}$. Clearly, $\varphi^{-1}\mathcal{A}\subseteq
\mathcal{A}\subseteq \Sigma_{\infty}$ and in this setting,
$E^{\mathcal{A}}(f)=\frac{f(x)+f(-x)}{2}$, which is the even part of
 $f\in L^p(\Sigma)$.
 Fix $r>1$ and define the weight function $u$ on
$\mathbb{R}$ by
$$u(x)=\left\{
  \begin{array}{ll}
    2 x+r, & {1\leq x}, \\
       -x^2 -\frac{x}{2}+2, & {-1< x < 1},\\
       x^3+ \frac{1}{r}, & {x\leq -1}.
  \end{array}
\right.$$ Then, we have
$$E^{\mathcal{A}}(u)(x)=\left\{\begin{array}{lll}
r,&1 \leq x,\\ -\frac{x^2}{2}+2, &-1< x < 1,\\
\frac{1}{r},&x \leq -1.
\end{array}\right.$$
 For an arbitrary $F=[-a, a]$, take
$V_k=(-a+\frac{1}{k}, a-\frac{1}{k})$.
    In this case, one may easily find a sequence $(n_k)$ such that both quantities
$\|(\prod_{i=0}^{n_k-1}E^{\mathcal{A}}(u)\circ\va^{i})^{-1}|_{V_k}\|_{\infty}$
  and $\|\sqrt[p]{h^{\mathcal{A}}_{n_k}}\ [\prod_{i=0}^{{n_k}-1}
  E^{\mathcal{A}}(u)\circ\va^{i}]\circ
   \va^{-n_k}|_{V_k}\|_{\infty}$ tend zero as $k\rightarrow \infty$.
Because, $h^{\mathcal{A}}_{n_k}=h^{\mathcal{A}^\sharp}_{n_k}=1$ and
$[\prod_{i=0}^{{n_k}-1}E^{\mathcal{A}}(u)\circ\va^{i}]\circ
   \va^{-n_k}=\prod_{i=1}^{{n_k}}E^{\mathcal{A}}(u)\circ\va^{-i}$,
   since $\varphi$ is onto (or $\sigma(h^{\mathcal{A}}_{n_k})=\mathbb{R}$).
    Therefore, by Theorem \ref{T11}, $T_u$ is
  subspace-hypercyclic with respect to
$L^p(\mathcal{A})$ while it is not hypercyclic on $L^p(\Sigma)$
\cite[Theorem 2.3]{az2}. For this, just consider that
$\|\sqrt[p]{h_{n_k}}[E_{n_k}( \prod_{i=0}^{{n_k}-1}
  u\circ\va^{i})]\circ
   \va^{-n_k}|_{V_k}\|_{\infty}=\| \prod_{i=1}^{{n_k}}
  u\circ\va^{-i}|_{V_k}\|_{\infty}\nrightarrow 0$.
\end{exam}

\end{document}